\documentclass[reqno,10pt]{amsart}
\pdfoutput=1

\usepackage[utf8]{inputenc}
\usepackage[english]{babel}
\usepackage{amsfonts}
\usepackage{amssymb}
\usepackage{bbm}
\usepackage{tikz}
\usepackage{svg}
\usepackage[right=2cm,bottom=2cm,footskip=30pt,left=2.5cm]{geometry}
\usepackage[autostyle]{csquotes}
\MakeOuterQuote{"}
\usepackage{color}
\usepackage{mathrsfs}
\usepackage{mathtools,leftidx}
\usepackage{enumitem}
\usepackage{amsmath}
\usepackage{mathrsfs}
\usepackage{accents}
\usepackage{scalerel}
\usepackage[most]{tcolorbox}
\usepackage{aligned-overset}
\setlist[enumerate,2]{label={(\theenumi.\theenumii)},ref={(\theenumi.\theenumii)},leftmargin=1cm}
\setlist[enumerate,1]{label={(\theenumi)},ref={(\theenumi)},leftmargin=1cm}

\usetikzlibrary{svg.path,arrows,calc}
\tikzset{>=latex'}

\usepackage[toc]{appendix}
\usepackage{etoolbox}
\cslet{blx@noerroretextools}\empty
\usepackage[maxbibnames=9,giveninits=true,style=alphabetic,sorting=nyt,backend=biber]{biblatex}
\DeclareFieldFormat{titlecase:title}{\MakeSentenceCase*{#1}}
\renewbibmacro*{url+urldate}{%
\iffieldundef{urlyear}
  {}
  {}
\usebibmacro{url}}
\def\restrict#1{\raise-.5ex\hbox{\ensuremath|}_{#1}}

\usepackage[colorlinks, citecolor=citegreen, linkcolor=refred,urlcolor=blue, unicode,psdextra,hypertexnames=false]{hyperref}
\usepackage{cleveref}
\crefname{enumi}{}{}
\crefname{enumii}{}{}
\expandafter\def\csname ver@etex.sty\endcsname{3000/12/31}

\makeatletter
 \def\author@andify{
 \nxandlist {\unskip{\kern.3cm} \penalty-2}
 {\unskip {\kern.3cm} \penalty-2}
 {\unskip {\kern.3cm} \penalty-2}}
\makeatother

\newcommand{\orcid}[1]{\unskip {} \raisebox{.4ex}{\href{https://orcid.org/#1}{\resizebox{.25cm}{.25cm}{\includegraphics{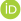}}}}}

\usepackage{accents}
\usepackage{autonum}
\usepackage{gasymbols}
\usepackage{nicefrac}

\definecolor{citegreen}{rgb}{0,0.3,0}
\definecolor{refred}{rgb}{0.5,0,0}

\usepackage[equation=section]{theorems}

\let\oldemail\email
\let\email\relax
\def\email#1{\oldemail{\href{mailto:#1}{\textcolor{black}{#1}}}}

\title{The equivalence of isocapacitary notions of mass}

\author[L.~Benatti]{Luca Benatti\orcid{0000-0002-4685-7443}}
\address{L.~Benatti, University of Vienna,
Oskar-Morgenstern-Platz 1, 1090, Vienna, Austria}
\email{luca.benatti@unive.ac.at}

\renewcommand{\ncapa}{\mathfrak{c}}


\newcommand{\sml}[1]{%
  \mathchoice%
    {#1}% displaystyle
    {#1}% textstyle
    {\scaleto{(#1)}{5pt}}% scriptstyle
    {\scaleto{(#1)}{4pt}}% scriptscriptstyle
}

\begin{document}
 
\begin{abstract}
In this short note, we will prove the equivalence of the isocapacitary notions of mass. This family also includes G. Huisken's isoperimetric mass and  J. L. Jauregui's isocapacitary mass.
\end{abstract}  
\maketitle

\noindent MSC (2020): 
53C21, %Methods of global Riemannian geometry, including PDE methods; curvature restrictions [See also 58J60] Recent zbMATH articles in MSC 53C21  5374
53E10, %FLOW RELATED TO MEAN CURVATURE
38C35, %Integration on manifolds; measures on manifolds;
49J45. %Methods involving semicontinuity and convergence; relaxation

\smallskip

\noindent \underline{\smash{Keywords}}: concepts of mass, isoperimetric mass, isocapacitary mass, nonnegative scalar curvature.

\section{Introduction}
Consider a smooth complete Riemannian $3$-manifold $(M,g)$. In a previous work with M. Fogagnolo and L. Mazzieri \cite{benatti_nonlinear_2023}, we introduced a one-parameter family of global masses.
\begin{definition}[Isocapacitary masses]\label{def:masses}
Fix $p \in [1,3)$. We define the \textit{quasilocal iso-$p$-capacitary mass} for any bounded $\Omega \Subset M$ as
\begin{equation}\label{eq:QLisopmass}
    \ma^{\sml{p}}_{\iso}(\Omega)\coloneqq \frac{1}{2 \pi p \ncapa_p(\partial \Omega)^{\frac{2}{3-p}}} \left( \abs{\Omega}- \frac{4\pi}{3} \ncapa_p(\partial \Omega)^{\frac{3}{3-p}} \right),
\end{equation}
where $\ncapa_p(\partial \Omega)$ is the $p$-capacity, namely
\begin{equation}\label{eq:pcapacity}
    \ncapa_p(\partial \Omega) \coloneq \inf\set{\frac{1}{4\pi}\left(\frac{p-1}{3-p}\right)^{p-1}\int_{M\smallsetminus \Omega}\abs{\nabla \varphi}^{p}\dif \mu\st \varphi\in \CS^\infty_c(M),  \varphi \geq 1 \text{ on }\Omega}.
\end{equation}
The \textit{iso-$p$-capacitary mass} of $(M,g)$ is then defined as
\begin{equation}\label{eq:isopmass}
    \ma^{\sml{p}}_{\iso} \coloneqq \sup_{(\Omega_j)} \limsup_{j\to +\infty} \ma^{\sml{p}}_{\iso}(\Omega_j),
\end{equation}
where the supremum is taken over all exhaustions $(\Omega_j)_{j \in \N}$ of bounded subsets in $M$.
\end{definition}

This definition is not entirely new. The case $p=2$ corresponds to the \textit{isocapacitary mass} introduced by J. L. Jauregui in \cite{jauregui_adm_2023}. For $p=1$, $\ma^{\sml{p}}_{\iso}$ coincides with the \textit{isoperimetric mass} $\ma_{\iso}$ introduced by G. Huisken \cite{huisken_isoperimetric_2009}. While the first assertion is rather immediate, the latter requires some additional remarks. First, we implicitly assume $(p-1)^{(p-1)} \coloneqq1$ when $p=1$ in \cref{eq:pcapacity}. Second, $\ma^{\sml{1}}_{\iso} \geq \ma_{\iso}$ follows directly from $4\pi \ncapa_1(\partial \Omega) \leq \abs{\partial \Omega}$. Although the reverse inequality is true in general, it becomes evident under the additional assumption that every $\Omega$ admits a bounded strictly outward minimizing hull $\Omega^*$. In such case, $\ncapa_1(\partial \Omega)= \abs{\partial \Omega^*}/(4\pi)$ and so $\ma^{\sml{1}}_{\iso}(\Omega) \leq \ma_{\iso}(\Omega^*)$. We postpone the proof of the equivalence without this assumption to \cref{prop:equivalence1mas}.

In manifolds $(M,g)$ with nonnegative scalar curvature, the family of masses introduced in \cref{def:masses} circumvents some of the limitations inherent to the ADM mass $\ma_{\ADM}$ \cite{arnowitt_coordinate_1961}. Indeed, the expression of the latter involves first-order derivatives of the metric coefficients, and thus requires them to be differentiable at least in some weak sense. Moreover, even for smooth metrics, the value of this mass may depend on the choice of the coordinate chart in which the coefficients are computed. R. Bartnik \cite{bartnik_mass_1986} and P. T. Chru\'sciel \cite{chrusciel_boundary_1986} proved that this is not the case whenever $(M,g)$ is $\CS^1_{\tau>1/2}$-asymptotically flat. We say that $(M,g)$ is $\CS^k_{\tau}$-asymptotically flat if there exists a coordinate chart $x=(x^1, x^2,x^3):M\smallsetminus K \to \R^3 \smallsetminus \set{\abs{x} \leq r}$ for some compact $K$ and $r>0$, and $g_{ij}(x) = \delta_{ij} + O_k(\abs{x}^{-\tau})$. Isocapacitary masses rely on remarkably milder assumptions. They require solely the concepts of perimeter, capacity, and volume. Therefore, they extend naturally to more general settings, such as manifolds equipped with continuous metrics -- a case beyond the scope of this note, yet indicative of a promising direction for future research. Moreover, they are manifestly global geometric invariants, as their definitions involve no reference to local coordinates. 

The mathematical expression in \cref{eq:QLisopmass} is designed to recover the mass of any Schwarzschild spatial manifold with nonnegative mass when evaluated on its cross-sections. Therefore, the resulting isocapacitary mass coincides with the ADM mass in this model scenario. Remarkably, this property remains valid in a more general setting. The equivalence $\ma_{\iso} = \ma_{\ADM}$ was established by J. L. Jauregui and D. A. Lee \cite{jauregui_lower_2019}, and independently by O. Chodosh, M. Eichmair, Y. Shi, and H. Yu \cite{chodosh_isoperimetry_2021} through a different argument. Both results require nonnegative scalar curvature but also stronger assumptions on the metric than those needed to define the ADM mass. However, the core of the technique developed in \cite{jauregui_lower_2019} can be exploited to extend this equivalence to the full generality allowed, as shown in the paper with M. Fogagnolo and L. Mazzieri \cite{benatti_isoperimetric_2025}. For the other isocapacitary masses, the equivalence $\ma^{\sml{p}}_{\iso}= \ma_{\ADM}$ was proved for $p=2$ by J. L. Jauregui \cite{jauregui_adm_2023} in manifolds with nonnegative scalar curvature and additional asymptotic assumptions on the metric. In \cite{benatti_nonlinear_2023}, we reestablished the result in the full generality permitted -- as in the case of the isoperimetric mass -- and extended it to all $p\in(1,3)$.

\medskip

The previous discussion highlights that all isocapacitary masses are all equivalent whenever the ADM mass is defined. A natural question then arises: what happens to this equivalence when the ADM mass is no longer in play? The strongest result we have been able to prove so far is that $\ma^{(p)}_{\iso} \to \ma_{\iso}$ as $p \to 1^+$ \cite{benatti_nonlinear_2023}. 

\begin{center}
\begin{tcolorbox}[frame empty, colback=gray!20, halign=center, width=.8\textwidth, sharp corners,boxsep=1mm, before skip=.2cm,after skip=.1cm]
    \textit{We assume throughout the paper that $(M,g)$ is a smooth, complete, connected,}\\ \textit{orientable, one-ended and noncompact Riemannian $3$-manifold.}
\end{tcolorbox}
\end{center}

The goal of this note is to establish the following theorem.

\begin{theorem}\label{thm:maintheorem}
    Let $(M,g)$ be a Riemannian $3$-manifold with nonnegative scalar curvature. Suppose that $M$ possibly has a smooth compact minimal boundary and no other compact minimal surface is contained in $M$. Assume that $M$ satisfies an Euclidean isoperimetric inequality, namely
    \begin{equation}\label{eq:isoperimetric_inequality}
        \exists\, \kst_I>0 \quad \text{such that} \quad \abs{\partial E}^3 \geq \kst_I \abs{E}^2 \quad \forall\, E \Subset M.
    \end{equation}
    Then, $\ma^{\sml{p}}_{\iso} = \ma_{\iso}$ for all $p\in [1,3)$.
\end{theorem}

The case $p=1$ will be addressed separately in \cref{prop:equivalence1mas}, since the equivalence follows directly from the definition. For all other $p$'s, we divide the theorem into two parts. In \cref{sec:greatest}, we will prove that the isoperimetric mass is the smallest among the isocapacitary masses. The proof builds on an asymptotic isoperimetric inequality and on a symmetrization argument, similar to \cite{benatti_nonlinear_2023}. In particular, neither the curvature assumption nor \cref{eq:isoperimetric_inequality} is needed.  In \cref{sec:smallest}, we will prove the reverse inequality. Since the proof is based on a Geroch-type monotonicity formula and \cite{jauregui_lower_2019}, all the assumptions in \cref{thm:maintheorem} are required. 
\medskip

\textbf{Acknowledgments.} This research was funded in part by the Austrian Science Fund (FWF) [grant DOI \href{https://www.fwf.ac.at/en/research-radar/10.55776/EFP6}{10.55776/EFP6}]. For open access purposes, the author has applied a CC BY public copyright license to any author-accepted manuscript version arising from this submission.

The author gratefully acknowledges support from the Simons Center for Geometry and Physics, Stony Brook University. Some of the research for this paper was carried out during the workshop \textit{Geometric Measure Theory on Metric Spaces with Applications to Physics and Geometry}, which was part of the thematic program \textit{Geometry and Convergence in Mathematical General Relativity}.

The author thanks M. Fogagnolo, whose insightful questions greatly improved this final version compared to the initial draft. The author is also grateful to C. Sinestrari for the intensive email exchange on topics related to this note. Thanks are also due to D. Carazzato, J. Jauregui, R. Perales, and A. Pluda for their valuable discussions and comments on the manuscript. 

\section{The case \texorpdfstring{$p=1$}{p=1}}\label{sec:p=1}
We start by proving the equivalence between the isoperimetric mass and the iso-$1$-capacitary mass. The statement does not require any particular assumption. Among capacities, the $1$-capacity is the closest to the notion of perimeter. Indeed, for any $\Omega\subseteq M$ bounded, we have
    \begin{equation}\label{eq:1cap}
    \ncapa_1(\partial \Omega) = \inf \set{ \frac{1}{4\pi}\abs{\partial E} \st\Omega\subseteq E \Subset M \text{ with smooth boundary}}.
    \end{equation}

\begin{proposition}\label{prop:equivalence1mas}
    Let $(M,g)$ be a Riemannian $3$-manifold. Then, $\ma^{\sml{1}}_{\iso}= \ma_{\iso}$.
\end{proposition}

\begin{proof}
    Fix an exhaustion $(\Omega_j)_{j \in \N}$. We distinguish two cases. 
    
    Assume that $\ncapa_1(\partial \Omega)=0$ for all subset $\Omega \Subset M$. Consequently, $\ncapa_1(\partial \Omega_j)=0$ for all $j \in \N$, then $\ma^{\sml{1}}_{\iso}=+\infty$. By \cref{eq:1cap}, for each $j \in \N$ there exists a set $E_j \supseteq \Omega_j$ such that $\abs{\partial E_j} \to 0$ as $j \to +\infty$. In particular, $\ma_{\iso}(E_j) \to +\infty$ as $j\to +\infty$. Since $(E_j)_{j \in \N}$ is again an exhaustion, it follows that $\ma_{\iso}=+\infty$.

    Assume now there exists $\Omega\Subset M$ with positive $1$-capacity. Without loss of generality, we may assume $\Omega_j \supseteq \Omega$ and therefore $\ncapa_1(\partial \Omega_j) \geq \ncapa_1(\partial \Omega)>0$. Let $(\varepsilon_j)_{j \in \N}$ be such that $\varepsilon_j \in (0,1)$ and $\ncapa_1(\partial \Omega_j)\varepsilon_j \to 0 $ as $j \to +\infty$. By \cref{eq:1cap}, for all $j \in \N$ there exists a $E_j\supseteq \Omega_j$ such that $ 4\pi \ncapa_1(\partial \Omega_j)\geq (1- \varepsilon_j)\abs{\partial E_j}$. As before, $(E_j)_{j \in \N}$ is again an exhaustion. Moreover,
    \begin{equation}\label{eq:zzzbound1cap}
        \ma_{\iso}^{\sml{1}}(\Omega_j) \leq \frac{2}{(1-\varepsilon_j)\abs{\partial E_j}}\left(\abs{E_j}- \frac{(1-\varepsilon_j)^{\frac{3}{2}}\abs{\partial E_j}^{\frac{3}{2}}}{6 \sqrt{\pi}}\right) \leq \frac{\ma_{\iso}(E_j)}{1- \varepsilon_j} + \frac{\varepsilon_j\abs{\partial E_j}^{\frac{1}{2}}}{1- \varepsilon_j},
    \end{equation}
    which immediately yields  $\ma^{\sml{1}}_{\iso} \leq \ma_{\iso}$. The reverse inequality follows directly from $4 \pi \ncapa_1(\partial \Omega_j) \leq \abs{\partial \Omega_j}$, as mentioned in the introduction.
\end{proof}

\section{The isoperimetric mass is the biggest}\label{sec:greatest}
We are now going to prove the following proposition.
\begin{proposition}\label{prop:secondinequality}
    Let $(M,g)$ be a strongly $p$-nonparabolic Riemannian $3$-manifold with possibly empty compact boundary. Then $\ma_{\iso} \geq \ma_{\iso}^{\sml{p}}$ for every $p \in (1,3)$.
\end{proposition}

Let $p \in (1,3)$. A Riemannian manifold $(M,g)$ is called \textit{strongly $p$-nonparabolic} if there exists a proper \textit{weak $p$-inverse mean curvature flow} ($p$-IMCF for short) with compact initial condition $\Omega$, which is a solution $w_p:M\smallsetminus \Omega \to \R$ to the following problem
\begin{equation}\label{eq:pIMCF}
    \begin{cases}
        \displaystyle\div\left(\frac{\nabla w_p}{\abs{\nabla w_p}^{2-p}}\right) &=& \abs{\nabla w_p}^p & \text{on }M \smallsetminus \Omega,\\
        w_p&=& 0 &\text{on } \partial \Omega,\\
        w_p(x) &\to& + \infty & \text{as }\mathrm{dist}(x,\partial \Omega) \to + \infty.
    \end{cases}
\end{equation}
in the sense that $u_p \coloneqq \ee^{-w_p/(p-1)}$ is the $p$-capacitary potential of the condenser $(\Omega, M)$. Spaces with a positive isoperimetric constant are examples of strongly $p$-nonparabolic manifolds, as shown for example in \cite[Theorem 3.3]{benatti_proper_2025}. Hence, \cref{prop:secondinequality} applies to the setting of \cref{thm:maintheorem}. Moreover, by the maximum principle, a large class of subsets $\Omega$ admits a solution to \cref{eq:pIMCF}.

\begin{definition}
    We say that a compact $\Omega\Subset M$ is a \textit{coexterior domain} if $M\smallsetminus \Omega$ is a domain whose boundary is a closed $\CS^{1,1}$-hypersurface and each connected component of $\partial M$ is either contained in or disjoint from $\Omega$.
\end{definition}

If $\Omega \supseteq \partial M$, the solution to \cref{eq:pIMCF} is unique. Conversely, if $\partial M \smallsetminus \Omega \neq \varnothing$, different solutions may be constructed. For instance, one may impose $w_p = 0$ on $\partial M \smallsetminus \Omega$ or, alternatively, one may compactly fill in the remaining boundary components of $\partial M$ and solve the problem on the resulting Riemannian manifold. However, such an ambiguity never arises in the present note.

\smallskip

The proof of \cref{prop:secondinequality} follows the same strategy as \cite[Theorem 5.5]{benatti_nonlinear_2023}. Besides including it for the sake of completeness, we also do so to highlight and fix a minor flaw in the original argument. In fact, a step in the proof requires the isoperimetric mass to be nonnegative\footnote{One can use \cite[(5.11)]{benatti_nonlinear_2023} to bound \cite[(5.9)]{benatti_nonlinear_2023} from above only if $\ma_{\iso} \geq 0$.}. Nevertheless, the conclusion remains valid since asymptotic flatness is assumed: the isoperimetric mass is nonnegative under this assumption, as shown in the subsequent paper by J. L. Jauregui, D. A. Lee, and R. Unger \cite{jauregui_note_2024}. However, \cref{prop:secondinequality} does not require any asymptotic behavior of the metric, and thus a refined argument becomes necessary.

The next lemma isolates the part of the argument leading to the asymptotic isoperimetric inequality \cite[(5.6)]{benatti_nonlinear_2023}. Here, however, the inequality is stated for large sets in the sense of inclusion rather than volume. This formulation is in fact the required one, and it does not require asymptotic flatness.

\begin{lemma}\label{lem:sharpisoperimetric}
    Assume that $\ma_{\iso} <+ \infty$. Then for every $\ma > \ma_{\iso}$ there exists a compact subset $K_{\ma}$ such that for every $\Omega \supseteq K_{\ma}$ it holds 
    \begin{equation}\label{eq:sharpisoperimetric}
        \abs{\Omega} \leq \frac{\abs{\partial \Omega}^{\frac{3}{2}}}{6 \sqrt{\pi}} + \frac{\ma}{2}\abs{\partial \Omega}.
    \end{equation}
\end{lemma}

\begin{proof}
    We argue by contradiction. Suppose that there exists $\ma> \ma_{\iso} $ such that, for every $K$ one can find $\Omega \supseteq K$ for which
    \begin{equation}
        \abs{\Omega} > \frac{\abs{\partial \Omega}^{\frac{3}{2}}}{6 \sqrt{\pi}} + \frac{\ma}{2}\abs{\partial \Omega}.
    \end{equation}
    In particular, $\ma_{\iso}(\Omega) \geq \ma$. Hence, we can build an exhaustion of $(K_j)_{j \in \N}$ such that $\ma_{\iso}(\Omega_j)\geq \ma$ for some compact $\Omega_j \supseteq K_j$. Since $(\Omega_j)_{j\in \N}$ is still an exhaustion, we conclude that $\ma_{\iso} \geq \ma > \ma_{\iso}$.
\end{proof}

\begin{proof}[Proof of \cref{prop:secondinequality}]
Assume $\ma_{\iso}<+\infty$, otherwise the statement is trivial. Let $\ma>\ma_{\iso}$ and $K_{\ma}\supseteq \partial M$ as in \cref{lem:sharpisoperimetric}.
Since $M$ is $p$-nonparabolic, its volume is not finite. Hence, we can assume that $\abs{K_{\ma}}$ is sufficiently large, so that $\abs{\partial \Omega} \geq 1$. In particular, the isoperimetric inequality
\begin{equation}\label{eq:euclideanisoperimetric}
    \abs{\Omega} \leq \left( \frac{1}{6\sqrt{\pi}} + \frac{\abs{\ma}}{2}\right) \abs{\partial \Omega}^{\frac{3}{2}}
\end{equation}
holds for every $\Omega \supseteq K_{\ma}$. In particular, we can choose $\abs{K_{\ma}}$ as big as we want in the following steps, so that both $\abs{\Omega}$ and $\abs{\partial \Omega}$ are sufficiently large for any $\Omega \supseteq K_{\ma}$.

Fix now any $\Omega\supseteq K_{\ma}$ and evolve it using the solution $w_p:M \smallsetminus \Omega \to \R$ to the problem \cref{eq:pIMCF}. Denote $\Omega_t \coloneqq \set{w_p\leq - (p-1) \log t }$ and $V(t) \coloneqq \abs{\Omega_t \smallsetminus \set{\abs{\nabla w_p}=0}}$. By H\"older inequality, we get that
\begin{equation}\label{eq:area-pcapinequality}
    \abs{\partial \Omega_t}^p \leq \left(\int_{\partial \Omega_t} \frac{1}{\abs{\nabla w_p}}\dif \sigma\right)^{p-1}\left(\int_{\partial \Omega_t} \abs{\nabla w_p}^{p-1}\dif \sigma\right) \leq \ncapa_p[-V'(t)]^{p-1},
\end{equation}
where we denoted
\begin{equation}
    \ncapa_p\coloneqq 4 \pi \left( \frac{3-p}{p-1}\right)^{p-1}\ncapa_p(\partial \Omega).
\end{equation}

Let $R(t)$ be such that $V(t)= 4 \pi R(t)^3/3$ and $v:\set{\abs{x} \geq R(1)} \subseteq \R^3\to (0,1] $ be such that $\set{v=t} = \set{\abs{x}=R(t)}$ for every $t \in (0,1]$. Observe that $v =1 $ on $\set{\abs{x}=R(1)}$. Moreover, $v$ is Lipschitz since 
\begin{align}
    \abs{\nabla v} &=- 4\pi \left( \frac{3}{4 \pi}\right)^{\frac{2}{3}} \frac{V(t)^{\frac{2}{3}}}{V'(t)} \overset{\cref{eq:area-pcapinequality}}{\leq}
     4\pi \left( \frac{3}{4 \pi}\right)^{\frac{2}{3}} \ncapa_p^{\frac{1}{p-1}}\frac{V(t)^{\frac{2}{3}}}{\abs{\partial \Omega_t}^{\frac{p}{p-1}}} \\    
     \overset{\cref{eq:euclideanisoperimetric}}&{\leq}  \left(\frac{1}{6\sqrt{\pi}} + \frac{\abs{\ma}}{2}\right)^{\frac{2p}{3(p-1)}}\ncapa_p^{\frac{1}{p-1}}V(t)^{-\frac{2}{3(p-1)}}
\end{align}
and $V(t) \geq \abs{K_{\ma}}$. Using $v$ as a competitor in the definition of the $p$-capacity of $\set{\abs{x} = R(1)}$, coarea formula, and Jensen's inequality, we then have
\begin{align}
    \abs{\Omega} &= \frac{4 \pi}{3} \ncapa_p(\set{\abs{x} \leq R(1)})^{\frac{3}{3-p}}\leq \frac{(4 \pi )^{-\frac{p}{3-p}}}{3}\left( \frac{p-1}{3-p}\right)^{\frac{3(p-1)}{3-p}}\left[\int_{\set{\abs{x} \geq R(1)}} \abs{\nabla v}^{p} \dif x\right]^{\frac{3}{3-p}}\\
    &=\frac{(4 \pi )^{-\frac{p}{3-p}}}{3}\left( \frac{p-1}{3-p}\right)^{\frac{3(p-1)}{3-p}}\left[\int_0^1\int_{\set{v=t}} \abs{\nabla v}^{p-1} \dif \sigma\dif t\right]^{\frac{3}{3-p}}\\
    & = 3^{\frac{3(p-1)}{(3-p)}}\left( \frac{p-1}{3-p}\right)^{\frac{3(p-1)}{3-p}}\left[\int_0^1\frac{[V(t)]^{\frac{2p}{3}}}{[-V'(t)]^{p-1}}\dif t\right]^{\frac{3}{3-p}}\\
    & = 3^{\frac{3(p-1)}{(3-p)}}\left( \frac{p-1}{3-p}\right)^{\frac{3(p-1)}{3-p}}\int_0^1\frac{[V(t)]^{\frac{2p}{3-p}}}{[-V'(t)]^{\frac{3(p-1)}{3-p}}}\dif t.\label{eq:zzzlastinequalitypolyaszegov}
\end{align}
  The sharp isoperimetric inequality in \cref{lem:sharpisoperimetric} yields
  \begin{equation}\label{eq:zzzlowerboundisoperimetricmneg}
        [6 \sqrt{\pi} V(t)]\left( 1- \frac{3\sqrt{\pi}\ma}{[6 \sqrt{\pi}V(t)]^{\frac{1}{3}}}\right)\overset{ }\leq \abs{\partial \Omega_t}^\frac{3}{2}.
    \end{equation}
Moreover, the function 
        \begin{equation}
            [(6 \sqrt{\pi}\abs{K_{\ma}})^\frac{2}{3},+\infty) \ni\abs{\partial \Omega_t}  \mapsto \left(\frac{1}{(\sqrt{\abs{\partial \Omega_t}})} \right)^{\frac{3-p}{p-1}}  \left( 1 + \frac{3\ma \sqrt{\pi}  }{\sqrt{\abs{\partial \Omega_t}}}\right)
        \end{equation}
    is monotone nondecreasing, provided $\abs{K_\ma}$ is sufficiently large. Hence,
    \begin{align}
             6 \sqrt{\pi} V(t)\overset{\cref{eq:sharpisoperimetric}}&{\leq}\abs{\partial \Omega_t}^{\frac{3}{2}+ \frac{3-p}{2(p-1)}} \left(\frac{1}{(\sqrt{\abs{\partial \Omega_t}})} \right)^{\frac{3-p}{p-1}}  \left( 1 + \frac{3\ma \sqrt{\pi}  }{\sqrt{\abs{\partial \Omega_t}}}\right) \\    \overset{\cref{eq:area-pcapinequality}}&{\leq}[-\ncapa_p^{\frac{1}{p-1}} V'(t)]\left(\frac{1}{(\sqrt{\abs{\partial \Omega_t}})} \right)^{\frac{3-p}{p-1}}  \left( 1 + \frac{3\ma \sqrt{\pi}  }{\sqrt{\abs{\partial \Omega_t}}}\right)\\
             \overset{\cref{eq:zzzlowerboundisoperimetricmneg}}&{\leq}\frac{[-\ncapa_p^{\frac{1}{p-1}} V'(t)]}{[6 \sqrt{\pi} V(t)]^{\frac{3-p}{3(p-1)}}}\left[1- \frac{3\ma \sqrt{\pi}}{[6 \sqrt{\pi} V(t)]^\frac{1}{3}}\right]^{-\frac{3-p}{3(p-1)}}\left[1+ \frac{3 \ma\sqrt{\pi}}{[6 \sqrt{\pi} V(t)]^{\frac{1}{3}}}\left( 1 - \frac{3\ma\sqrt{\pi}}{[6 \sqrt{\pi} V(t)]^\frac{1}{3}}\right)^{-\frac{1}{3}}\right].\\
        \end{align}
        Writing the first-order Taylor polynomial of the function
        \begin{equation}
            \frac{3 \ma\sqrt{\pi}}{[6 \sqrt{\pi} V(t)]}\mapsto\left[1- \frac{3\ma \sqrt{\pi}}{[6 \sqrt{\pi} V(t)]^\frac{1}{3}}\right]^{-\frac{2p}{3(p-1)}}\left[1+ \frac{3 \ma\sqrt{\pi}}{[6 \sqrt{\pi} V(t)]^{\frac{1}{3}}}\left( 1 - \frac{3\ma\sqrt{\pi}}{[6 \sqrt{\pi} V(t)]^\frac{1}{3}}\right)^{-\frac{1}{3}}\right]^{\frac{2p}{3-p}}
        \end{equation}
        and estimating the remainder, we have that
        \begin{equation}\label{eq:zzz-mneg-differntial}
            [6 \sqrt{\pi} V(t)]^{\frac{2p}{3-p}}\leq  \frac{[-\ncapa_p^{\frac{1}{p-1}} V'(t)]^{\frac{2p}{3-p}}}{[6 \sqrt{\pi} V(t)]^{\frac{2p}{3(p-1)}}}\left[ 1+ \frac{4 p^2}{3(3-p)(p-1)}\frac{3\ma\sqrt{\pi}}{[6 \sqrt{\pi} V(t)]^\frac{1}{3}}+k_p \frac{9\ma^2\pi }{[6 \sqrt{\pi} V(t)]^{\frac{2}{3}}}\right],
        \end{equation}
        for some $k_p>0$ depending only on $p$. Integrating \cref{eq:zzz-mneg-differntial} and changing the variable, one obtains
        \begin{align}
            \int_0^1\frac{[V(t)]^{\frac{2p}{3-p}}}{[-V'(t)]^{\frac{3(p-1)}{3-p}}}\dif t &\leq \frac{\ncapa_p^{\frac{3}{3-p}+ \frac{1}{p-1}}}{(6 \sqrt{\pi})^\frac{4p^2}{3(p-1)(3-p)}}\int_{\abs{\Omega}}^{+\infty}V^{-\frac{2p}{3(p-1)}} \left( 1+ \frac{4p^2}{3(3-p)(p-1)} \frac{3 \ma\sqrt{\pi}}{[6 \sqrt{\pi}V]^{\frac{1}{3}}} +k_p\frac{9\ma^2 \pi}{[6 \sqrt{\pi}V]^\frac{2}{3}}\right)\dif V\\
            &= \frac{3(p-1)}{(3-p)}\frac{\ncapa_p^{\frac{3}{3-p}+ \frac{1}{p-1}}}{(6 \sqrt{\pi})^\frac{4p^2}{3(p-1)(3-p)}}\abs{\Omega}^{-\frac{3-p}{3(p-1)}}\left[1+\frac{2p^2}{3(p-1)}\frac{3 \ma\sqrt{\pi}}{[6 \sqrt{\pi}\abs{\Omega}]^{\frac{1}{3}}} + k'_p \frac{9\ma^2 \pi}{[6 \sqrt{\pi}\abs{\Omega}]^\frac{2}{3}}\right],
        \end{align}
        where $k'_p>0$ depends again only on $p$. Plugging it into \cref{eq:zzzlastinequalitypolyaszegov}, we have the following iso-$p$-capacitary estimate
        \begin{equation}
            \abs{\Omega}^{\frac{2p}{3(p-1)}}\leq \left( \frac{4 \pi}{3}\right)^{\frac{2p}{3(p-1)}} \ncapa_p(\partial \Omega)^{\frac{3}{3-p}+ \frac{1}{p-1}}\left[1+\frac{2p^2}{3(p-1)}\frac{3 \ma\sqrt{\pi}}{[6 \sqrt{\pi}\abs{\Omega}]^{\frac{1}{3}}} + k'_p \frac{9\ma^2}{[6 \sqrt{\pi}\abs{\Omega}]^\frac{2}{3}}\right],
        \end{equation}
        and then
        \begin{equation}\label{eq:finaleestimate}
            \abs{\Omega} \leq\frac{4 \pi}{3} \ncapa_p(\partial \Omega)^{\frac{3}{3-p}}\left[1+ \frac{3p\ma \sqrt{\pi} }{[6 \sqrt{\pi} \abs{\Omega}]^{\frac{1}{3}}} + k'_p \frac{18\ma^2 \pi}{[6 \sqrt{\pi}\abs{\Omega}]^\frac{2}{3}}\right].
        \end{equation}

We now distinguish two cases.
\begin{case}[$\ma\leq0$]
        The function
        \begin{equation}
            [\abs{K_{\ma}},+\infty)\ni \abs{\Omega} \mapsto \left[1+ \frac{3p\ma \sqrt{\pi} }{[6 \sqrt{\pi} \abs{\Omega}]^{\frac{1}{3}}} + k'_p \frac{18\ma^2 \pi}{[6 \sqrt{\pi}\abs{\Omega}]^\frac{2}{3}}\right]\leq 1 
        \end{equation}
        is monotone nondecreasing. Hence, \cref{eq:finaleestimate} yields
        \begin{equation}
            \abs{\Omega} \leq \frac{4 \pi}{3}\ncapa_p(\partial \Omega),
        \end{equation}    
        that plugged back into \cref{eq:finaleestimate} gives
        \begin{equation}\label{eq:finalmneg}
          \abs{\Omega} \leq\frac{4 \pi}{3} \ncapa_p(\partial \Omega)^{\frac{3}{3-p}}\left[1+ \frac{3p\ma }{2\ncapa_p(\partial \Omega)^{\frac{1}{3-p}}} + k'_p \frac{9\ma^2 \pi}{2\ncapa_p(\partial \Omega)^{\frac{2}{3-p}}}\right]  .
        \end{equation}
\end{case}
\begin{case}[$\ma> 0$] 
If
\begin{equation}
    \abs{\Omega} \leq \frac{4\pi}{3} \ncapa_p(\partial \Omega),
\end{equation}
then $\ma^{\sml{p}}_{\iso}(\Omega) \leq 0 < \ma$. Otherwise, \cref{eq:finaleestimate} yields
\begin{equation}\label{eq:finalmpos}
          \abs{\Omega} \leq\frac{4 \pi}{3} \ncapa_p(\partial \Omega)^{\frac{3}{3-p}}\left[1+ \frac{3p\ma }{2\ncapa_p(\partial \Omega)^{\frac{1}{3-p}}} + k'_p \frac{9\ma^2 \pi}{2\ncapa_p(\partial \Omega)^{\frac{2}{3-p}}}\right]  .
        \end{equation}
\end{case}

Take now an exhaustion $(\Omega_j)_{j \in \N}$ such that $\Omega_j \supseteq K_{\ma}$. Rearranging either \cref{eq:finalmneg,eq:finalmpos}, we can show that
        \begin{equation}\label{eq:zzz-almostdone-mneg}
            \limsup_{j \to +\infty}\ma^{\sml{p}}_{\iso}(\Omega_j) \leq \limsup_{j\to +\infty}\ma + 3 k'_p\ma^2 \ncapa_p(\partial \Omega_j)^{-\frac1{3-p}}= \ma,
        \end{equation}
        which implies $\ma^{\sml{p}}_{\iso} \leq \ma$. Since $\ma>\ma_{\iso}$ is generic, we conclude the proof.
\end{proof}

\section{The isoperimetric mass is the smallest}\label{sec:smallest}
We are now going to prove the following proposition.
\begin{proposition}\label{prop:firstinequality}
    Let $(M,g)$ be a Riemannian $3$-manifold with nonnegative scalar curvature. Suppose that $M$ possibly has a smooth compact minimal boundary and no other compact minimal surface is contained in $M$. Assume that \cref{eq:isoperimetric_inequality} holds. Then, $\ma_{\iso} \leq \ma_{\iso}^{\sml{p}}$ for every $p \in (1,3)$.
\end{proposition}

The proof of \cref{prop:firstinequality} crucially builds on \cite[Theorem 17]{jauregui_lower_2019}. As already noted in \cite{benatti_isoperimetric_2025}, many of the assumptions in that theorem are included to ensure that the argument leading to \cite[Main Theorem]{huisken_inverse_2001} can be carried out. If one removes the ADM mass from the picture, \cite[Theorem 17]{jauregui_lower_2019} effectively requires an upper bound on the \textit{Hawking mass} for a sufficiently large class of surfaces. The Hawking mass of a closed $\CS^2$-surface $\Sigma$ is defined as
\begin{equation}
    \ma_{H}(\Sigma) \coloneqq\sqrt{\frac{\abs{\Sigma}}{16 \pi}}\left(1- \int_{\Sigma} \frac{\H^2}{16\pi} \dif \sigma\right),
\end{equation}
where $\H$ is the mean curvature of $\Sigma$. The following statement is essentially \cite{jauregui_lower_2019}, tailored for our purposes.

\begin{theorem}[\cite{jauregui_lower_2019}]\label{prop:jaureguilee}
    Let $(M,g)$ be a Riemannian $3$-manifold. Suppose that $M$ possibly has a smooth compact minimal boundary and no other compact minimal surface is contained in $M$. Assume that \cref{eq:isoperimetric_inequality} holds. If $\ma_H(\partial \Omega) \leq \ma$ for every coexterior domain $\Omega$ with connected boundary, then $\ma_{\iso} \leq \ma$.
\end{theorem}

\begin{proof}
    Take an exhaustion $\set{\Omega_j}_{j \in \N}$ of $M$. If the isoperimetric ratio $i_j \coloneqq \abs{\partial \Omega_j}^{3/2}/\abs{\Omega_j}$ diverges to $+\infty$, then $\ma_{\iso}(\Omega_j) \to -\infty$ as $j\to +\infty$. Therefore, either $\ma_{\iso}=-\infty$ or we can assume that $i_j$ is bounded.
    
    Observe that all the assumptions required in the proof of \cite[Theorem 17]{jauregui_lower_2019} are now met with $\ma$ in place of the ADM mass. In particular, nonnegative scalar curvature and asymptotic flatness are required only to ensure the upper bound for the Hawking mass in the class of coexterior domains. The isoperimetric constant is bounded by \cref{eq:isoperimetric_inequality}, which also ensures that $\abs{\partial \Omega_j} \geq 36 \pi \ma^2$ for $j$ large enough. Hence,
    \begin{equation}
        \ma_{\iso}(\Omega_j) \leq \ma + \frac{\kst }{\sqrt{\abs{\partial \Omega_j}}},
    \end{equation}
    which implies $\ma_{\iso}\leq \ma$.
\end{proof}

In our setting, it suffices to control the Hawking mass on a narrower family of domains. By \cite[Theorem 1.2]{benatti_proper_2025}, spaces with positive isoperimetric constants are \textit{strongly $1$-nonparabolic}, meaning that every coexterior domain $\Omega\supset \partial M$ admits a weak inverse mean curvature flow (IMCF) $w_1: M \smallsetminus \Omega \to \R$, i.e., a locally Lipschitz solution to the following problem
\begin{equation}\label{eq:IMCF}
    \begin{cases}
        \displaystyle\div\left(\frac{\nabla w_1}{\abs{\nabla w_1}}\right) &=& \abs{\nabla w_1} & \text{on }M \smallsetminus \Omega,\\
        w_1&=& 0 &\text{on } \partial \Omega,\\
        w_1(x) &\to& + \infty & \text{as }\mathrm{dist}(x,\partial \Omega) \to + \infty.
    \end{cases}
\end{equation}
The differential equation in \cref{eq:IMCF} is understood in the sense of Huisken--Ilmanen \cite{huisken_inverse_2001}. 

In manifolds with nonnegative scalar curvature, the map $t \mapsto \ma_H(\partial \set{w_1 \le t})$ is monotone nondecreasing, provided that level sets $\partial \set{w_1 \le t}$ remain connected (see \cite[Geroch Monotonicity Formula 5.8]{huisken_inverse_2001} or \cite[Theorem 5.5]{benatti_fine_2024}). The condition on minimal surfaces guarantees this topological property if the initial data $\Omega \supseteq \partial M$ has connected boundary\footnote{\cite[Remark 2.12]{benatti_isoperimetric_2025} ensures that \cite[Lemma~2.11]{benatti_isoperimetric_2025} applies. In the second item of that lemma, the assumption that the exhaustion consists of spheres can be removed, as one can verify by inspecting the proof. The lemma also states that \cref{prop:firstinequality} -- and \cref{thm:maintheorem} -- also apply when minimal surfaces are confined in a compact set.}. In these cases, an upper bound for the Hawking mass can be derived by analyzing its asymptotic behavior along the level sets of solutions to \cref{eq:IMCF}. This observation often simplifies the task of ensuring the hypothesis $\ma_{H}(\partial \Omega)\leq \ma$ in \cref{prop:jaureguilee}, since the geometry at infinity is typically better controlled than the local one. Moreover, the monotonicity can be used to further restrict the class of coexterior domains $\Omega$ to those that entirely contain the boundary $\partial M$. In the setting of \cref{thm:maintheorem}, we can refine \cref{prop:jaureguilee} as follows.

\begin{lemma}\label{lem:jaureguirefined}
Under the hypotheses of \cref{thm:maintheorem}, denote $\Omega_t \coloneqq \set{w_1 \le t}$, where $w_1$ is the solution of \cref{eq:IMCF} with initial data $\Omega$. If $\lim\limits_{t \to +\infty} \ma_H(\partial \Omega_t) \le \ma $ for every coexterior domain $\Omega\supseteq \partial M$ with connected boundary, then $\ma_{\iso} \le \ma$.
\end{lemma}

\begin{proof}
    We only need to prove that $\ma_{H}(\partial \Omega) \leq \ma$ holds for every coexterior domain with connected boundary. We have already discussed how monotonicity formulas imply $\ma_{H}(\partial \Omega) \leq \ma_{H}(\partial \Omega_t) \leq \ma$ whenever $\Omega \supset \partial M$. 
    
    On the contrary, assume that at least one connected component of $\partial M$ is disjoint from $\Omega$. By \cite[Geroch Monotonicity Formula 6.1]{huisken_inverse_2001} (see also \cite[Proposition 2.6]{benatti_isoperimetric_2025}), one can build a coexterior domain $\Omega'\supset \partial M$ with connected boundary such that $\ma_{H}(\partial \Omega) \leq \ma_{H}(\partial \Omega')$. Therefore, the latter case reduces to the previous one.
\end{proof}

The lack of regularity of the IMCF is one of the main difficulties in analyzing the asymptotic behavior of quantities computed on its level sets. If the background metric itself lacks asymptotic behavior, the issue becomes even more severe. The following result -- extracted from the asymptotic comparison lemma \cite[Lemma 2.8]{benatti_isoperimetric_2025} --  allows control of the $L^2$-norm of the mean curvature on large level sets under the sole assumption that the Hawking mass remains bounded.

\begin{lemma}\label{lem:contwillmore}
    Under the hypotheses of \cref{prop:firstinequality}, let $\Omega \supseteq \partial M$ be a coexterior domain with connected boundary. Denote $\Omega_t \coloneqq \set{w_1 \le t}$, where $w_1$ is the solution of \cref{eq:IMCF} with initial data $\Omega$. If $\sup_t\ma_{H}(\partial \Omega_t)<+\infty$, then
    \begin{equation}\label{eq:contwillmore}
        \lim_{t \to +\infty} \int_{\partial \Omega_t} \H^2 \dif \sigma = 16 \pi.
    \end{equation}
\end{lemma}
\begin{proof}
  By the monotonicity of the Hawking mass, either $\ma_H(\partial \Omega_t)<0$ for all $t$ or there exists $T\geq0$ such that $\ma_H(\partial \Omega_t) \geq 0 $ for all $t\geq T$. In the first case, we would have
  \begin{equation}
      \int_{\partial \Omega_t} \H^2  \dif \sigma > 16\pi
  \end{equation}
  for all $t>0$. Assume by contradiction that the lemma is false, then there exists a sequence $(t_n)_{n \in \N}$ increasing to $+ \infty$ as $n\to + \infty$ such that
    \begin{equation}
        \int_{\partial \Omega_t} \H^2 \dif \sigma \geq 16\pi + 16 \pi \varepsilon,
    \end{equation}
    for some $\varepsilon>0$.
    Hence,
    \begin{equation}
        \ma_H(\partial \Omega) \leq \ma_H(\partial \Omega_{t_n}) \leq - \varepsilon \ee^{\frac{t_n}{2}} \sqrt{\frac{\abs{\partial \Omega^*}}{16 \pi}}
    \end{equation}
    which contradicts $\H \in L^2(\partial \Omega)$. Thus, \cref{eq:contwillmore} holds if $\ma_{H}(\partial \Omega_t)<0$ for all $t$.
    \smallskip

    Assume now that there exists $T\geq0$ such that $\ma_{H}(\partial \Omega_t)\geq 0$ for all $t \geq T$. Arguing by contradiction, pick a sequence $(t_n)_{n\in\N}$ increasing to $+\infty$ as $n \to +\infty$ and such that 
    \begin{equation}
        \int_{\partial \Omega_{t_n}} \H^2 \dif \sigma \leq 16\pi - 16 \pi \varepsilon,
    \end{equation}
    for some positive $\varepsilon\leq 1 $. It follows that
    \begin{equation}
        \ma_{H}(\partial \Omega_{t_n}) \geq \varepsilon \sqrt{\frac{\abs{\partial \Omega_{t_n}}}{16\pi}},
    \end{equation}
    which contradicts $\sup\ma_{H}(\partial \Omega_{t_n}) <+\infty$. Hence, \cref{eq:contwillmore} holds also in this case, concluding the proof of the lemma.
\end{proof}

The last ingredient is the following theorem proved for $p=2$ by H. Bray and P. Miao \cite{bray_capacity_2008} and then extended for any $p \in (1,3)$ by J. Xiao \cite{xiao_p-harmonic_2016}. We remark that asymptotic flatness in \cite{xiao_p-harmonic_2016,bray_capacity_2008} is assumed only to ensure that $(M,g)$ is strongly $1$-nonparabolic.

\begin{theorem}[\cite{bray_capacity_2008},\cite{xiao_p-harmonic_2016}]\label{prop:xiao}
Let $(M,g)$ be a Riemannian $3$-manifold with nonnegative scalar curvature. Suppose that $M$ possibly has a smooth compact minimal boundary and no other compact minimal surface is contained in $M$. Assume that \cref{eq:isoperimetric_inequality} holds. Then, 
    \begin{equation}
        \ncapa_p(\partial \Omega)  \leq \left( \frac{\abs{\partial \Omega}}{4\pi}\right)^{\frac{3-p}{2}} \leftidx{_2}{F}{_1}\left( \frac{1}{2}, \frac{3-p}{p-1},\frac{2}{p-1}; 1- \frac{1}{16 \pi}\int_{\partial \Omega} \H^2 \dif \sigma \right)^{-(p-1)}
    \end{equation}
holds for every coexterior domain $\Omega\supseteq \partial M$ with connected boundary, where $ \leftidx{_2}{F}{_1}$ is the hypergeometric function.
\end{theorem}

The hypergeometric function in the theorem satisfies the following identity
\begin{equation}
     \leftidx{_2}{F}{_1}\left( \frac{1}{2}, \frac{3-p}{p-1},\frac{2}{p-1}; 1-t\right) = \frac{3-p}{p-1}(1-t)^{-\frac{3-p}{p-1}}\int_t^1 \frac{(1-s)^{\frac{2}{p-1}-2}}{\sqrt{s}} \dif s, \qquad \text{for }t \in (-\infty, 1].
\end{equation}
All the properties used in the proof below follows follows from this identity.

\begin{proof}[Proof of \cref{prop:firstinequality}]

    Let $\Omega \supseteq \partial M$ be a coexterior domain with connected boundary. Let $\Omega_t\coloneqq \set{w_1 \leq t}$. Assume that $\ma^{\sml{p}}_{\iso} <+\infty$, otherwise the statement is trivial. This implies that $\sup\ma_{H}(\partial \Omega_t) <+\infty$. If it is not the case, we would have $\ma_{H}(\partial \Omega_t) \geq 0$ for all $t \geq T$, by monotonicity. By \cref{prop:xiao}, this implies
    \begin{equation}
        \ncapa_p(\partial \Omega_t) \leq \left( \frac{\abs{\partial \Omega_t}}{4 \pi }\right)^{\frac{3-p}{2}}
    \end{equation}
    for all $t \geq T$. In particular, by \cite[Lemma 2.8]{benatti_isoperimetric_2025} we would have
    \begin{equation}
        \liminf_{t \to +\infty} \ma^{\sml{p}}_{\iso}(\partial \Omega_t) \geq \liminf_{t \to +\infty} \frac{1}{p}\ma_{\iso}(\Omega_t) \geq \lim_{t \to +\infty} \frac{1}{p}\ma_{H}(\partial \Omega_t),
    \end{equation}
    that contradicts the hypothesis on $\ma^{\sml{p}}_{\iso}$.  
    
    By \cref{lem:contwillmore} we have that
     \begin{equation}
         1- \frac{1}{16\pi}\int_{\partial \Omega_t} \H^2 \dif \sigma = o(1)
     \end{equation}
     as $t \to +\infty$. Since
     \begin{equation}\label{eq:firstorderG}
         \leftidx{_2}{F}{_1}\left( \frac{1}{2}, \frac{3-p}{p-1},\frac{2}{p-1};1-s \right) = 1+ \frac{3-p}{4}(1-s) +O((1-s)^2).
     \end{equation}
     as $s\to 1^+$, by algebraic computations involving \cref{prop:xiao} one can get
     \begin{equation}
         \ma^{\sml{p}}_{\iso}(\Omega_t) =\left[\frac{1}{p}\ma_{\iso}(\Omega_t) +\frac{p-1}{p}\ma_{H}(\partial \Omega_t)\right](1+o(1)), \qquad \text{as } t \to +\infty.
     \end{equation}
     Hence, employing again \cite[Lemma 2.8]{benatti_isoperimetric_2025}
     \begin{equation}
        \ma^{\sml{p}}_{\iso}\geq \liminf_{t \to +\infty} \ma^{\sml{p}}_{\iso} (\Omega_t) \geq \lim_{t \to +\infty} \ma_H(\partial \Omega_t) \geq \ma_H(\partial \Omega),
     \end{equation}
     which concludes the proof.
\end{proof}

\section{Final remarks}

It is noteworthy that \cref{prop:secondinequality} holds in a broader setting. However, it is not clear at this stage whether it is the broadest setting possible. The technique adopted involves the existence of solutions to \cref{eq:pIMCF}. It may also be possible that existence follows as a consequence of the finiteness of the masses. The finiteness of $\ma^{\sml{p}}_{\iso}$ implies at least that $(M,g)$ is $p$-nonparabolic, i.e. there exists a set of positive $p$-capacity, which is known to be equivalent to the existence of a positive $p$-Green function. If this function vanishes at infinity, then the manifold would be strongly $p$-nonparabolic.

\begin{question}
    Does $\ma^{\sml{p}}_{\iso} <+\infty$ imply that $(M,g)$ is strongly $p$-nonparabolic?
\end{question}

Another question closely related to the previous one is the following.

\begin{question}
    Does $\ma^{\sml{p}}_{\iso}=+\infty$ for some $p\in[1,3)$ implies that $\ma^{\sml{p}}_{\iso}=+\infty$ for all $p\in[1,3)$?
\end{question}

On the one hand, there are examples of spaces that are $p$-parabolic -- thus $\ma^{\sml{p}}_{\iso} =+\infty$ -- for all $p>1$, but $1$-nonparabolic. A negative answer to the question seems to be more spontaneous. On the other hand, it is not known weather being $1$-nonparabolic, or even strongly $1$-nonparabolic, implies $\ma_{\iso}<+\infty$. 

\medskip

It is also noteworthy that \cref{prop:secondinequality} is based purely on asymptotic analysis, while \cref{prop:firstinequality} requires also local assumptions. This difference reflects the fact that the isoperimetric mass is determined by the geometry inside the domain, whereas isocapacitary masses, $p>1$, depend on what happens outside of it. However, when passing to the limit along an exhaustion, only the geometry at infinity should matter. This idea becomes even clearer in light of \cref{prop:equivalence1mas}: this result suggests that the $1$-capacity is the true actor on the stage, rather than the perimeter.
\begin{question}
    Is the consequence of \cref{prop:firstinequality} valid in strongly $1$-nonparabolic manifolds?
\end{question}

\begingroup
\setlength{\emergencystretch}{1em}
\printbibliography
\endgroup

\end{document}